\documentclass{amsart}

\usepackage{amssymb}
\usepackage{amsmath}
\usepackage{amsthm}

\usepackage{graphicx}

\newtheorem {theorem}{Theorem}[section]

\newtheorem {cory}[theorem]{Corollary}
\newtheorem {lemma}[theorem]{Lemma}
\theoremstyle{definition}
\newtheorem{definition}{Definition}
\theoremstyle{remark}
\newtheorem{rem}{Remark}

\numberwithin{equation}{section}

\newcommand{\al}{\alpha}
\newcommand{\de}{\delta}
\newcommand{\la}{\lambda}

\newcommand{\za}{\zeta}
\newcommand{\ff}{w}
\newcommand{\FF}{\Phi}
\newcommand{\hg}{h}
\newcommand{\GG}{g}
\newcommand{\dd}{\beta}
\newcommand{\ee}{n}
\newcommand{\we}{w}
\newcommand{\tve}{\widetilde{v}}
\newcommand{\ve}{v}
\newcommand{\xx}{x^\prime}

\newcommand{\ph}{\varphi}
\newcommand{\cph}{C_\varphi}
\newcommand{\vph}{V_\varphi}
\newcommand{\hol}{\mathcal{H}ol}

\newcommand{\berg}{A}
\newcommand{\rom}{r_{\mathcal D}}
\newcommand{\dom}{\mathcal D}
\newcommand{\Dbb}{\mathbb D}
\newcommand{\YY}{\mathcal Y}
\newcommand{\spn}{\partial B_d}
\newcommand{\bd}{B_d}
\newcommand{\cd}{{\mathbb{C}}^d}
\newcommand{\Rbb}{\mathbb R}
\newcommand{\Cbb}{\mathbb C}
\newcommand{\Nbb}{\mathbb N}

\newcommand{\rad}{\mathcal R}

\newcommand{\grw}{\mathcal{A}^{\we}}

\newcommand{\jg}{J_g}

\begin{document}

\title[Moduli of holomorphic functions and log-convex radial weights]{Moduli of holomorphic functions and logarithmically  convex radial weights}

\author[E.~Abakumov]{Evgeny Abakumov}

\address{Universit\'e Paris-Est, LAMA (UMR 8050), UPEMLV, UPEC, CNRS,
F-77454, Marne-la-Vall\'ee, France}

\email{evgueni.abakoumov@u-pem.fr}

\author[E.~Doubtsov]{Evgueni Doubtsov}

\address{St.~Petersburg Department
of V.A.~Steklov Mathematical Institute, Fontanka 27, St.~Petersburg
191023, Russia;
\newline
\phantom{\quad} Department of Mathematics and Mechanics, St.~Petersburg State University,
Universitetski pr. 28, St.~Petersburg 198504, Russia}

\email{dubtsov@pdmi.ras.ru}

\thanks{The second author was supported by the Russian Science Foundation (grant No.~14-41-00010).}

\subjclass[2010]{Primary 30H99; Secondary 32A37, 42A55, 47B38}

\date{}

\keywords{Growth space, log-convex radial weight, circular domain}

\begin{abstract}
Let $\hol(\Dbb)$ denote the space of holomorphic functions
on the unit disk $\Dbb$.
We characterize those radial weights $\we$ on $\Dbb$
for which there exist functions $f, g\in\hol(\Dbb)$
such that the sum $|f| + |g|$ is equivalent to $\we$.
Also, we obtain similar results in several complex variables for
circular, strictly convex domains with smooth boundary.
\end{abstract}

\maketitle

\section{Introduction}\label{s_int}

\subsection{Growth spaces}\label{ss_growth_df}
Let $\hol(\Dbb)$ denote the space of holomorphic functions
on the unit disk $\Dbb$. Consider a weight function $\we$,
that is, a non-decreasing, continuous, unbounded function
$\we: [0, 1)\to (0, +\infty)$.
We extend $\we$ to a radial weight on $\Dbb$ setting
$\we(z) = \we(|z|)$, $z\in\Dbb$.
By definition,
the growth space $\grw(\Dbb)$ consists of $f\in\hol(\Dbb)$
such that
\begin{equation}\label{e_phi}
|f(z)|\le C\we(z), \quad z\in \Dbb,
\end{equation}
for some constant $C>0$.

Given functions $u, v: \Dbb\to (0,+\infty)$, we write $u\asymp v$
and we say that $u$ and $v$ are equivalent if
\[
C_1 u(z) \le v(z) \le C_2 u(z), \quad z\in\Dbb,
\]
for some constants $C_1, C_2 >0$.
The definition of equivalent weight functions is analogous.
If $\we_1$ and $\we_2$ are equivalent radial weights,
then the identities
\[
\|f\|_{\mathcal{A}^{\we_j}(\Dbb)} = \sup_{z\in\Dbb} \frac{|f(z)|}{\we_j(z)},
\quad j=1,2,
\]
define equivalent norms on the Banach space
$\mathcal{A}^{\we_1}(\Dbb) = \mathcal{A}^{\we_2}(\Dbb)$.

\subsection{Motivations}\label{ss_mot}
In various applications, it is useful to have test functions
$f\in\grw(\Dbb)$ for which the reverse of estimate (\ref{e_phi}) holds,
in a sense.
The starting point for the present paper is the following reverse estimate obtained by
W.~Ramey and D.~Ullrich~\cite{RU91} with the help of lacunary series:

\begin{theorem}[cf.\ {\cite[Proposition~5.4]{RU91}}]\label{t_RU}
Let $\we(t)=\frac{1}{1-t}$, $0\le t< 1$.
There exist functions $f_1, f_2 \in \grw(\Dbb)$ such that
\[
|f_1(z)| + |f_2(z)| \ge \we(|z|)
\]
for all $z\in \Dbb$.
\end{theorem}

One may consider Theorem~\ref{t_RU} as a particular solution of the
following approximation problem:

Given a radial weight $\we$ on $\Dbb$, find
$f_1, f_2 \in\hol(\Dbb)$ such that
\[
|f_1| + |f_2| \asymp w.
\]
In other words, we are looking for a holomorphic mapping $f: \Dbb\to \Cbb^2$
such that $\|f\| \asymp w$.

The above problem has been solved recently for various explicit radial weights;
see, for example, \cite{Ga08}, \cite{GX02}, \cite{GPPR08}, \cite{KS09}, \cite{Xi04}.
Clearly, the required property still holds if $\we$
is replaced by an equivalent radial weight.
To the best of our knowledge, the largest class of weight functions $\we(t)$ is
considered in \cite{AD12}, where
the direct analog of Theorem~\ref{t_RU} is proved under assumption that $\we(t)$
has the following doubling property:
\begin{equation}\label{e_reg_df_we}
\we(1 - s/2) < A \we(1 - s),
\quad 0<s\le 1,
\end{equation}
for some constant $A >1$.
Essentially the same result was independently obtained in \cite{KP11}.
As in \cite{RU91}, the arguments in \cite{AD12}, \cite{KP11}
use lacunary series;
see \cite{GPR14} for a different proof.

Basically, property~\eqref{e_reg_df_we} means that
$\we(t)$ 
grows sufficiently slowly
as $t\to 1-$.
So, it is natural to ask whether a growth restriction is
crucial for the corresponding results.
In the present paper, we characterize those $\we$
for which the approximation problem under consideration is solvable.
In particular, analogs of Theorem~\ref{t_RU}
hold for rapidly growing radial weights.
However, there are 
natural restrictions on the regularity of $\we$.

\subsection{Main results}\label{ss_int_model}
Let $\we: [0, 1)\to (0, +\infty)$ be a weight function.
Often, $\we$ is called log-convex if
$\log w$ is convex.
In this paper, we use a different definition
related to Hadamard's three-circles theorem.
Namely, $\we$
is said to be {\it log-convex} if
$\log w(t)$ is a convex function of $\log t$, $0<t<1$.
If $\log w(t)$ is a convex function of $t$, then $\we$ is log-convex, but the converse is
not true; see Remark~\ref{r_logconv}.

\subsubsection{Unit disk}\label{sss_int_disk}
\begin{theorem}\label{l_disk}
Let $\we$ be a radial weight on $\Dbb$.
Then the following properties are equivalent:
\begin{align}
&\textrm{there exist\ }
f_1, f_2\in\hol(\Dbb)
\textrm{\ such that\ }
|f_1| + |f_2| \asymp \we;\label{e_phi_2disk}
\\
&\we(t)
\textrm{\ is equivalent to a log-convex weight function on\ }
[0, 1).\label{e_logc_disk}
\end{align}
\end{theorem}

In other words, the above theorem characterizes those $\we$
for which an analog of Theorem~\ref{t_RU} holds.

It is worth mentioning the ideas behind the proof of Theorem~\ref{l_disk}.
The implication \eqref{e_phi_2disk}$\Rightarrow$\eqref{e_logc_disk}
is deduced from Hadamard's three-circles theorem and properties of log-convex functions.
To prove the reverse implication, we use a geometric construction
that generates $f_1$ and $f_2$ as appropriate power series.
The construction is based on approximation of convex functions of one real variable by piecewise linear functions.
Similar ideas were used in Borichev's work \cite{Bo98} on Fock-type spaces.
Note that the series used in \cite{AD12}, \cite{KP11}, \cite{RU91}
are Hadamard lacunary.
By contrast, for rapidly growing weight functions, we obtain
weakly lacunary series: the corresponding frequencies
grow slower than any geometric progression.

\begin{rem}
Condition \eqref{e_logc_disk} is not a growth restriction:
given a weight function $v$, there exists a log-convex weight function $\we$
such that $\we(t)\ge v(t)$, $0\le t <1$.
As mentioned above, \eqref{e_logc_disk} is a natural regularity condition.
\end{rem}

\begin{rem}
Theorem~\ref{l_disk} guarantees that
property~\eqref{e_logc_disk}
also characterizes those radial weights
which are equivalent to $\max(|f_1|, |f_2|)$
for some $f_1, f_2\in\hol(\Dbb)$.
\end{rem}

\begin{rem}\label{r_diskM}
Given an integer $M$, one may consider the following problem:
Characterize those
radial weights $w$ for which there exist $f_1,\dots, f_M \in \hol(\Dbb)$
such that
\[
|f_1| +\dots + |f_M| \asymp \we.
\]
 The problem degenerates for $M=1$: by the maximum modulus principle,
 no radial weight is equivalent to the
modulus of a holomorphic function. For
$M=2, 3,\dots$, the solution of the problem does not depend on $M$ and it is
given by \eqref{e_logc_disk}; see Remark~\ref{r_aab2}.
\end{rem}

\subsubsection{Circular domains}\label{sss_int_main}
We extend Theorem~\ref{l_disk} to certain circular domains in $\cd$, $d\in\Nbb$.
In particular, an analog of Theorem~\ref{l_disk} holds for the unit ball of $\cd$.

Assume that
$\dom\subset\cd$ is a bounded, circular, strictly convex domain
with ${\mathcal C}^2$-boundary.
Let $\rom(z)$ denote the Minkowski functional on $\dom$, that is,
\[\rom(z) =
\inf\{\rho>0:\ z\in\rho\dom\}.
\]
If $\dom$ is the unit ball, then
$\rom(z)=|z|$.

Let $\hol(\dom)$ denote the space of holomorphic functions
on $\dom$. Given a weight function
$\we: [0, 1)\to (0, +\infty)$,
the growth space $\grw(\dom)$ consists of $f\in\hol(\dom)$
such that
\[
\|f\|_{\grw(\dom)} = \sup_{z\in\dom} \frac{|f(z)|}{\we(\rom(z))}<\infty.
\]
We obtain the following extension of Theorem~\ref{l_disk} to several complex variables:

\begin{theorem}\label{t_main}
Let $\dom\subset\cd$ be a bounded, circular, strictly convex domain
with ${\mathcal C}^2$-boundary.
If $\we: [0, 1) \to (0, +\infty)$ is a log-convex weight function,
then there exist functions
$f_m \in \grw(\dom)$, $1\le m \le M = M(\dom)$ such that
\begin{equation}\label{e_phi_2}
\sum_{m=1}^{M} |f_m(z)| \ge \we(\rom(z)),
\quad z\in \dom.
\end{equation}
Conversely, if \eqref{e_phi_2} holds for some $M\in\Nbb$ and some $f_1,\dots, f_M\in\hol(\dom)$,
then $\we(t)$ is equivalent to a log-convex weight function.
\end{theorem}

\begin{rem}
Recall that, by Theorem~\ref{l_disk} and Remark~\ref{r_diskM},
the optimal value of $M$ for $\Dbb$ is equal to $2$.
It would be interesting to find the optimal value of $M(\bd)$
for the unit ball $\bd$ of $\cd$, $d\ge 2$.
\end{rem}

\subsection{Applications}\label{ss_appli}
We consider composition and multiplication operators,
Volter\-ra type operators and extended Ces\`aro operators on $\grw(\dom)$.

\subsection{Organization of the paper}\label{ss_int_org}
Section~\ref{s_model} is 
devoted to the proof of Theorem~\ref{l_disk}.
Circular domains in $\cd$ are considered in Section~\ref{s_circ}.
Applications are discussed in the final Section~\ref{s_appl};
main applications are related to the unit ball of $\cd$, $d\ge 1$.

\subsection*{Acknowledgements}
The authors are grateful to Alexander Borichev
for useful discussions and elucidating comments.
Also, we are thankful to the anonymous referee
for helpful suggestions and constructive remarks.

\section{$\dom$ is the unit disk}\label{s_model}
In this section, we prove Theorem~\ref{l_disk}.

\subsection{Equivalence to a log-convex weight function
is necessary}\label{ss_aab}
Given a continuous function $u: \Dbb\to (0, +\infty)$, put
\[
M_u(r) = \max_{|z|=r} u(z), \quad 0\le r <1.
\]
Now, suppose that $\we$ is a radial weight,
$f_1, f_2\in\hol(\Dbb)$
and $|f_1|+|f_2|\asymp\we$.
Without loss of generality, we also assume that
$f_j(0)\neq 0$ for $j=1,2$.
By Hadamard's three circles theorem,
$M_{|f_1|}$ and $M_{|f_2|}$ are log-convex functions on the interval $[0, 1)$;
see, e.g., \cite[Ch.~6, Sect.~3.13]{Con78}.
Hence, the sum $M_{|f_1|} +M_{|f_2|} $ is also log-convex
(see, for example, \cite[p.~51]{KPS82}).

Observe that
\[
M_{|f_j|} \le M_{|f_1|+|f_2|} \le M_{|f_1|} + M_{|f_2|}, \quad j=1,2.
\]
Thus, we have the following equivalences on $[0, 1)$:
\[
\we \asymp M_{|f_1|+|f_2|} \asymp M_{|f_1|} + M_{|f_2|}.
\]
Since the function $M_{|f_1|}+M_{|f_2|}$ is log-convex,
the implication \eqref{e_phi_2disk}$\Rightarrow$\eqref{e_logc_disk}
is proved.

\begin{rem}\label{r_aab2}
Let $\we$ be a radial weight such that
\[
\we\asymp |f_1|+\dots +|f_K|
\]
for some $f_1, \dots, f_K \in \hol(\Dbb)$, $K\in\Nbb$.
Repeating the arguments used above for $K=2$,
we deduce that $\we(t)$ is equivalent to a log-convex weight function.
\end{rem}

The rest of the section is devoted to a constructive proof of the implication
\eqref{e_logc_disk}$\Rightarrow$\eqref{e_phi_2disk}.

\subsection{Log-convex weight functions: preliminaries}\label{ss_logc_prelim}
Let $\ff: [0,1) \to (0,+\infty)$
be a log-convex weight function.
Recall that, by definition, $\log\ff(t)$ is a convex function of $\log t$, that is,
\[
\FF(x)=\FF_\we(x)=\log \ff(e^x), \quad x\in (-\infty, 0),
\]
is a convex function.

\begin{rem}\label{r_logconv}
It is natural to compare the above property and the following one:
$\log\we(t)$ is a convex function of $t$.
The latter property implies that $\FF_\we$
is convex, as the composition of two
increasing convex functions.
The reverse implication does not hold.
Moreover, there exist log-convex
weight functions that are not even equivalent to the exponent of a
convex function.
\end{rem}

In what follows, we argue in terms of the function $\FF$.
Observe that $\ff(t)=\exp(\FF(\log t))$, $0< t < 1$.

To prove the implication \eqref{e_logc_disk}$\Rightarrow$\eqref{e_phi_2disk},
we may replace $\ff$ by an equivalent weight function.
So, without loss of generality, we assume that
$\FF$ is a strictly convex $\mathcal{C}^2$-function.
In particular, the tangent
to the graph of $\FF$ is unique
at each point $(x, \FF(x))$, $x<0$.
Also, below we repeatedly use the following property without explicit reference:
the slope of the tangent
to the graph of $\FF$ at $(x, \FF(x))$
is a strictly increasing function of $x\in(-\infty, 0)$.

Finally, note that $\log v(e^x)=\log a + \dd x$ for $v(t)=a t^\dd$,
$a>0$, $\dd>0$.
This observation allows to reduce the proof
of the implication \eqref{e_logc_disk}$\Rightarrow$\eqref{e_phi_2disk}
to certain manipulations with linear functions.

\subsection{Basic induction construction}\label{ss_basind}
Fix a number $x_0\in (-\infty, 0)$ and a parameter $h>0$.
By induction, we construct linear functions $\ell_k(x)$
and numbers $x_k\in (x_0, 0)$, $k=1,2,\dots$, such that

\begin{itemize}
\item
$\ell_k(x)$ is a tangent to the graph of $\FF(x)$;
\item
$\ell_k(x)$ intersects the graph of $\FF(x) -\hg$ at the points whose $x$-coordinates
are $x_{k-1}$ and $x_k$, $x_{k-1} < x_k$.
\end{itemize}
Note that $\ell_k(x)$ and $x_k$ are uniquely defined by the above properties.
So, the induction construction proceeds. See Figure~\ref{fig_1}.

Given $\ell_k(x)$, we define parameters $a_k>0$ and $\dd_k>0$ by the following identity:
\[
\ell_k(x) = \log a_k + \dd_k x, \quad x\in\Rbb.
\]
Note that the sequences $\{x_k\}_{k=0}^\infty$ and $\{\dd_k\}_{k=1}^\infty$
monotonically increase, $x_k \to 0$ and $\dd_k\to\infty$ as $k\to \infty$.
Also, we use the following brief notation:
$t_k = \exp(x_k)$, $k=0,1,\dots$. Hence, the positive numbers
$t_k$ monotonically increase to $1$ as $k\to \infty$.

Formally, the above construction works for any $\hg>0$.
In applications, we use a sufficiently large parameter $h$,
say, $h=2$, in the case of the unit disk.

\begin{figure}
  \includegraphics[width=0.8\textwidth]{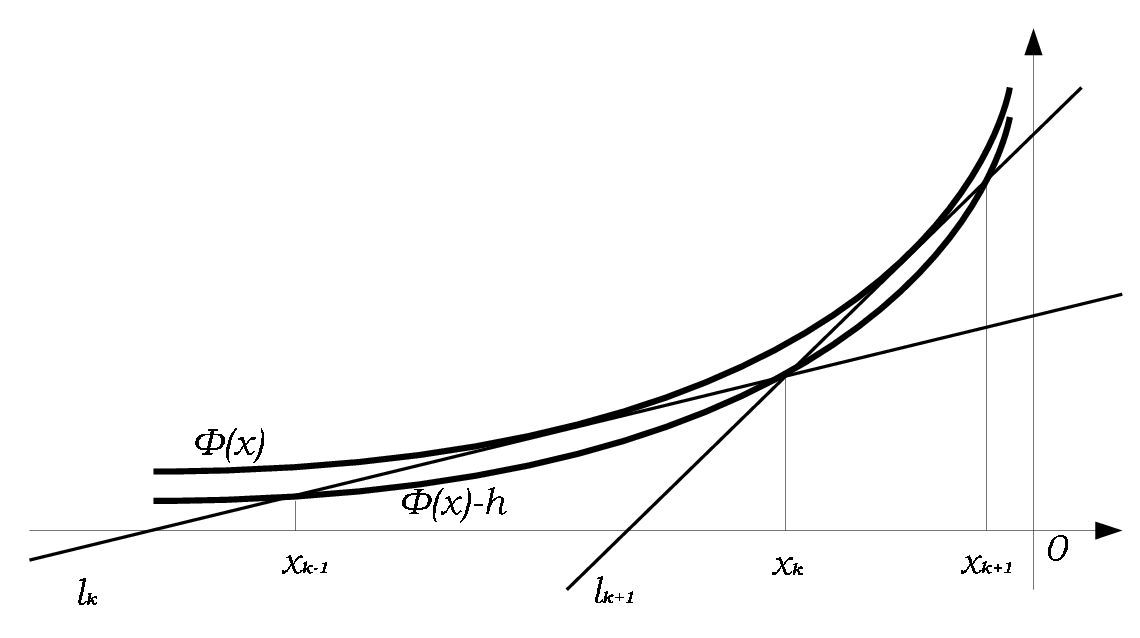}
\caption{Basic induction construction}
\label{fig_1}       
\end{figure}

\subsection{Auxiliary estimates}
\begin{lemma}\label{l_real_picture}
Let the numbers $x_k$, $k=0,1,\dots$,
and the linear functions $\ell_k$, $k=1,2,\dots$,
be those introduced in subsection~\ref{ss_basind}.
Then
\begin{align}
\ell_{k}(x) \ge \ell_{k+1}(x) +\hg \quad&\textrm{for all\ }x,\ x_0\le x\le x_{k-1},\ k\ge 1;
\label{e_picture_1}
\\
\ell_{k+1}(x) \ge \ell_{k}(x) +\hg \quad&\textrm{for all\ }x,\ x_{k+1}\le x< 0,\ k\ge 1.
\label{e_picture_2}
\end{align}
\end{lemma}
\begin{proof}
We verify property~\eqref{e_picture_1}.
The proof of \eqref{e_picture_2} is analogous, so we omit it.

For $k=1,2,\dots$, let $\xx_{k-1}$ denote the $x$-coordinate of the point
at which $\ell_{k}$ is tangent to the graph of $\FF$.
Remark that $x_{k-1} < \xx_{k-1} < x_{k}$ by the definition
of $\ell_{k}$.

We have $\ell_{k+1}(x_k)=\ell_k(x_k)$, $\dd_{k+1}>\dd_k$ and $x_{k-1}<\xx_{k-1}$,
thus,
\[
\ell_k(x) - \ell_{k+1}(x)\ge \ell_k(\xx_{k-1}) - \ell_{k+1}(\xx_{k-1})
\]
for all $x$, $x_0\le x \le x_{k-1}$.
Therefore, \eqref{e_picture_1} follows from the estimate
$\ell_k(\xx_{k-1}) - \ell_{k+1}(\xx_{k-1}) \ge \FF(\xx_{k-1}) - (\FF(\xx_{k-1})-h) =h$.
\end{proof}

\begin{lemma}\label{l_real_segment}
Let the numbers $x_k$, $t_k$, $k=0,1,\dots$,
and the linear functions $\ell_k$, $k=1,2,\dots$,
be those introduced in subsection~\ref{ss_basind}.
Assume that $\hg\ge 2$.
Then, for $k=1,2,\dots$,
\begin{itemize}
\item[(i)] 
$a_k t^{\dd_k} \le \ff(t),
\quad t\in [t_0, 1);$
\item[(ii)] 
$e^{-h} \ff(t)
\le a_k t^{\dd_k}, \quad t\in [t_{k-1}, t_k];$
\item[(iii)] 
$\sum_{m\ge 1,\ |m-k|\ge 2} a_m t^{\dd_m} < \frac{1}{2} a_k t^{\dd_k},
\quad t\in [t_{k-1}, t_k].$
\end{itemize}
\end{lemma}
\begin{proof}
Let $k\in\Nbb$.
We have $\ell_k(x) \le \FF(x)$ for $x_0\le x<0$.
Hence, taking the exponentials, we obtain (i)
by the definitions of $\FF$, $\ell_k$, $a_k$ and $\dd_k$.
Analogously, the inequality
$\FF(x)-\hg \le \ell_k(x)$, $x\in [x_{k-1}, x_k]$,
implies (ii).

It remains to prove (iii). 
First, fix a $k\ge 1$ and assume that
$m\ge k+2$. For $x_0 \le x \le x_k$, we have
\[
\ell_{k}(x) - \ell_m(x)
\ge \ell_{k+1}(x) - \ell_m(x)
\\
=\sum_{j=k+1}^{m-1} [\ell_j(x) -\ell_{j+1}(x)].
\]
Consider the latter sum.
For $k+1\le j\le m-1$,
property~\eqref{e_picture_1} guarantees that
$\ell_j(x) - \ell_{j+1}(x) \ge \hg$
for all $x_0\le x\le x_{j-1}$,
hence, for all $x_0\le x\le x_k$.
In sum, we have
\[
\ell_k(x) - \ell_m(x) \ge  \hg(m-k-1)
\quad\textrm{for\ } x_0\le x\le x_k.
\]
Taking the exponentials and using the definitions of $\ell_m$ and $\ell_k$,
we obtain
\[
a_m t^{\dd_m} \le a_k t^{\dd_k} \frac{1}{e^{\hg(m-k-1)}} \quad\textrm{for\ } t_0\le t\le t_k.
\]
Thus,
\begin{equation}\label{e_m_large}
\sum_{m\ge k+2} a_m t^{\dd_m} \le a_k t^{\dd_k}
\left( \frac{1}{e^{\hg}} + \frac{1}{e^{2\hg}} + \frac{1}{e^{3\hg}} + \dots\right)
 \quad\textrm{for\ } t_0\le t\le t_k.
\end{equation}

Second, fix $k\ge 3$ and assume that $1\le m \le k-2$.
Essentially, we argue as above, replacing \eqref{e_picture_1} by \eqref{e_picture_2}.
Namely, for $0> x\ge x_{k-1}$, we have
\[
\ell_{k}(x) - \ell_m(x)
\ge \ell_{k-1}(x) - \ell_m(x)
=\sum_{j=m}^{k-2} [\ell_{j+1}(x) -\ell_{j}(x)].
\]
For $m\le j \le k-2$, property~\eqref{e_picture_2} guarantees that
$\ell_{j+1}(x) - \ell_{j}(x) \ge \hg$ for all $0> x\ge x_{k-1}$.
In sum, we have
\[
\ell_m(x) \le \ell_k(x) - \hg (k-m-1)
\quad\textrm{for\ } 0> x\ge x_{k-1}.
\]
Taking the exponentials and summing over $m=1,\dots, k-2$,
we obtain
\begin{equation}\label{e_m_small}
\sum_{m=1}^{k-2} a_m t^{\dd_m} \le a_k t^{\dd_k}
\left( \frac{1}{e^{\hg}} + \frac{1}{e^{2\hg}} + \frac{1}{e^{3\hg}} + \dots\right)
\quad\textrm{for\ } 1> t\ge t_{k-1}.
\end{equation}

In sum, (iii) 
holds by \eqref{e_m_large} and \eqref{e_m_small},
because $\hg\ge 2$.
The proof of the lemma is complete.
\end{proof}

\subsection{Proof of \eqref{e_logc_disk}$\Rightarrow$\eqref{e_phi_2disk}}
Fix a parameter $h\ge 2$ and a number $x_0\in (-\infty, 0)$ such that $t_0=\exp(x_0)> \frac{9}{10}$ .
The basic induction construction provides numbers $x_k$, $t_k$, $\dd_k$ and $a_k$,
$k=1,2,\dots$.
Put $\ee_k = [\dd_k] +1$, $k=1,2,\dots$,
where $[\dd_k]$ denotes the integer part of $\dd_k$.
Note that $\ee_k < \ee_{k+1}$
for all $k=1,2,\dots$. Indeed, we have $|x_0|<1$, thus,
\begin{equation}\label{e_bebe}
\dd_{k+1} -\dd_k \ge \frac{h}{x_k - \xx_{k-1}} > \frac{h}{|x_0|} >1.
\end{equation}

So, put
\begin{align*}
\GG_1(z) &= \sum_{j=0}^\infty a_{2j+1} z^{\ee_{2j+1}};
\\
\GG_2(z) &= \sum_{j=1}^\infty a_{2j} z^{\ee_{2j}}.
\end{align*}
The estimates below guarantee that $\GG_1$ and $\GG_2$ are well-defined functions of $z\in\Dbb$.
In fact, we claim that
\begin{equation}\label{e_GG}
\frac{2}{5}e^{-h} \we(z) < |\GG_1(z)|+ |\GG_2(z)| <4\we(z),
\quad t_0 < |z| <1.
\end{equation}

To prove the above estimates, put $t=|z|$.
For $t\in\left(\frac{9}{10}, 1\right)$, we have
\[
1<\frac{a_k t^{\dd_k}}{a_k t^{\ee_k}} \le \frac{1}{t} \le \frac{10}{9}.
\]
Therefore, properties (i--iii) from Lemma~\ref{l_real_segment}
imply the following inequalities:
\begin{itemize}
\item[(i')] 
$a_k t^{\ee_k} \le \ff(t),
\quad t\in [t_0, 1);$
\item[(ii')] 
$\frac{9}{10} e^{-h} \ff(t)
\le a_k t^{\ee_k}, \quad t\in [t_{k-1}, t_k];$
\item[(iii')] 
\[
\begin{split}
\sum_{m\ge 1,\ |m-k|\ge 2} a_m t^{\ee_m}
&<\sum_{m\ge 1,\ |m-k|\ge 2} a_m t^{\dd_m}
\\
&<\frac{1}{2} a_k t^{\dd_k}
\le \frac{5}{9} a_k t^{\ee_k},
\quad t\in [t_{k-1}, t_k].
\end{split}
\]
\end{itemize}
Now, assume that $t\in [t_{k-1}, t_k]$
for some odd number $k$.
On the one hand,
\begin{align*}
\frac{2}{5}e^{-h} \we(t)
&\overset{\textrm{(ii')}}{<}
\frac{4}{9}a_{k} t^{\ee_{k}}
\\
&\overset{\textrm{(iii')}}{<}
a_{k} t^{\ee_{k}} - \sum_{m\ge 1,\ |m-k|\ge 2} a_{m} t^{\ee_{m}}
\\
&\le \left| \sum_{j=0}^\infty a_{2j+1} z^{\ee_{2j+1}} \right|
\\
&= |\GG_1(z)|
\\
&\le |\GG_1(z)| + |\GG_2(z)|.
\end{align*}
On the other hand, if $k\ge 3$, then
\begin{align*}
|\GG_1(z)| + |\GG_2(z)|
&\le \sum_{m\ge 1} a_m t^{\ee_m}
\\
&=\left(a_k t^{\ee_k} + \sum_{m\ge 1,\ |m-k|\ge 2} a_m t^{\ee_m}\right) +
\left(a_{k-1} t^{\ee_{k-1}} + a_{k+1} t^{\ee_{k+1}}\right)
\\
&\overset{\textrm{(i', iii')}}{\le}
\left(1+ \frac{5}{9}\right) a_k t^{\ee_k} + 2\we(t)
\\
&\overset{\textrm{(i')}}{<}4\we(t).
\end{align*}
For $k=1$, the above arguments are even more simple.

If $k$ is even, then we have analogous estimates. So, \eqref{e_GG} holds.

To finish the proof,
we modify the functions $\GG_1$ and $\GG_2$.
Namely, we obtain functions $f_1, f_2$ having the same property as $g_1, g_2$
but without common zeros.
So, we have $\GG_1(z) = z^{\ee_1} G_1(z)$,
where $G_1\in\hol(\Dbb)$ and $G_1(0)\neq 0$.
Clearly, $G_1\in \grw(\Dbb)$ and
$|G_1(z)|\ge |\GG_1(z)|$ for all $z\in\Dbb$.
Therefore,
\[
\frac{2}{5}e^{-h}\we(z) < |G_1(e^{i\theta}z)|+ |g_2(z)|,
\quad t_0 < |z| <1,
\]
for all $\theta\in [0, 2\pi)$.
The functions $G_1$ and $\GG_2$ have a finite number of zeros in the disk $\{|z|\le t_0\}$
and $G_1(0)\neq 0$; hence, the required property
\[
|f_1(z)| + |f_2(z)| \asymp \we(z),
\quad z\in \Dbb,
\]
holds for $f_2(z)=\GG_2(z)$ and $f_1(z)=G_1(e^{i\theta}z)$ with an appropriate $\theta\in [0, 2\pi)$.
So, \eqref{e_logc_disk} implies \eqref{e_phi_2disk}.
The proof of Theorem~\ref{l_disk} is finished.

\begin{rem}
As mentioned in the introduction, the sequence $\{n_k\}_{k=1}^\infty$
is not always lacunary in the classical Hadamard sense.
In fact, $\{n_k\}$ may grow as $C k^\alpha$ for certain $C, \alpha>0$.
We do not use this fact, so we omit its proof.
However, the series
\[
\sum_{k=1}^\infty \frac{1}{|\dd_{k+1} - \dd_k|}
\]
converges by \eqref{e_bebe}; in other words, the sequence $\{n_k\}_{k=1}^\infty$ retains
certain weak lacunarity for any weight function $\we$ under consideration.
\end{rem}

\section{$\dom$ is a circular domain}\label{s_circ}

We prove Theorem~\ref{t_main} in this section,
extending the arguments used for the unit disk.
So, $\dom\subset\cd$ is assumed to be bounded and circular.
Further restrictions ($\dom$ is a strictly convex domain
with ${\mathcal C}^2$-boundary)
come from Theorem~\ref{t_kot} that provides building blocks
in several complex variables.

\subsection{Equivalence to a log-convex weight function}
This question reduces to the corresponding implication in the unit disk.
Indeed, assume that \eqref{e_phi_2} holds.
Fix a point $\za\in\partial\dom$.
For $m=1,\dots, M$, consider the slice-functions
$g_m(\la) = f_m(\la\za)$, $\la\in\Dbb$. We have $g_m\in\hol(\Dbb)$, $m=1,\dots, M$,
and
\[
|g_1(\la)| +\dots+ |g_M(\la)| \asymp \we(\rom(\la\za)) = \we(|\la|), \quad \la\in\Dbb.
\]
Hence, as indicated in Remark~\ref{r_aab2},
$\we(t)$ is equivalent to a log-convex weight function.

\subsection{Aleksandrov--Ryll--Wojtaszczyk polynomials}\label{ss_ARW}
The proof of the existence part in Theorem~\ref{t_main} will be based on series
of special homogeneous holomorphic polynomials with sufficiently sparse degrees (cf.\
\cite{Aab86}, \cite{RW83}).

\begin{theorem}\label{t_kot}
Let $\dom\subset\cd$ be a bounded, circular, strictly convex domain
with ${\mathcal C}^2$-boundary. There exist $\de=\de(\dom)\in
(0,1)$ and $Q=Q(\dom)\in\Nbb$ with the following properties: for every
$n\in\Nbb$, there exist homogeneous holomorphic polynomials
$W_{q}[n]$ of degree $n$, $1\le q \le Q$, such that
\begin{align}
 \|W_q[n]\|_{L^\infty(\partial\dom)}
&\le 1; \label{e_RW_max} \\
 \max_{1\le q \le Q} |W_{q}[n](\za)|
&\ge \de\quad \textrm{for all\ } \za\in\partial\dom.
\label{e_RW_min}
\end{align}
\end{theorem}

As observed in \cite{DouCRAS}, to prove Theorem~\ref{t_kot},
it suffices to repeat \textit{mutatis mutandis} the arguments used in
\cite[Theorem~2.6]{Ko07}.

\subsection{Auxiliary estimates}
We need a modification of property~(iii) from Lem\-ma~\ref{l_real_segment}.

\begin{lemma}\label{l_real_segment_delta}
Let $\delta\in (0,1)$. Then there exists $h(\delta)\ge 2$
with the following property:
Let the numbers $x_k$, $t_k$, $k=0,1,\dots$,
and the linear functions $\ell_k$, $k=1,2,\dots$,
be those introduced in subsection~\ref{ss_basind}.
Assume that $\hg\ge h(\delta)$.
Then
\begin{itemize}
\item[(iii$_\delta$)] 
$\sum_{m\ge 1,\ |m-k|\ge 2} a_m t^{\dd_m} < \frac{\delta}{2} a_k t^{\dd_k},
\quad t\in [t_{k-1}, t_k]$,
\end{itemize}
for $k=1,2,\dots$.
\end{lemma}
\begin{proof}
The argument coincides with that used in the proof of Lemma~\ref{l_real_segment}.
Indeed, if $h\ge 2$ is sufficiently large, then
properties \eqref{e_m_large} and \eqref{e_m_small}  imply (iii$_\delta$).
\end{proof}

\subsection{Proof of Theorem~\ref{t_main}: existence part}\label{ss_tmain_prf}
Let $\ff: [0,1) \to (0,+\infty)$
be a log-convex weight function
and let
$\FF(x)=\log \ff(e^x)$, $x\in (-\infty, 0)$.

We modify the arguments used in the case of the unit disk,
replacing the monomials $z^n$, $n=1,2,\dots$, by
the Aleksandrov--Ryll--Wojtaszczyk polynomials $W_q[n]$,
$q=1,\dots, Q$, $n=1,2,\dots$.

So, let the constant $\delta\in (0,1)$ be that
provided by Theorem~\ref{t_kot}
and let the constant $h(\delta)\ge 2$ be that provided
by Lemma~\ref{l_real_segment_delta}.
Fix a parameter $h\ge h(\delta)$ and a number $x_0\in (-\infty, 0)$ such that $t_0=\exp(x_0)> \frac{9}{10}$ .
The basic induction construction provides numbers $x_k$, $t_k$, $\dd_k$ and $a_k$,
$k=1,2,\dots$.

For $s=0,1$, put
\[
f_{q+sQ}(z) = \sum_{j=0}^\infty a_{2j+1+s} W_q[n_{2j+1+s}](z), \quad
z\in\dom,\ q=1,\dots, Q.
\]
where $\ee_k = [\dd_k] +1$, $k=1,2,\dots$.

We adapt the argument used
to prove \eqref{e_GG}.
So, let $t=\rom(z)$.
Properties (i--ii) from Lemma~\ref{l_real_segment}
and property (iii$_\delta$) from Lemma~\ref{l_real_segment_delta}
imply the following inequalities:
\begin{itemize}
\item[(i')] 
$a_k t^{\ee_k} \le \ff(t),
\quad t\in [t_0, 1);$
\item[(ii')] 
$\frac{9}{10} e^{-h} \ff(t)
\le a_k t^{\ee_k}, \quad t\in [t_{k-1}, t_k];$
\item[(iii$_\delta$')] 
\[
\begin{split}
\sum_{m\ge 1,\ |m-k|\ge 2} a_m t^{\ee_m}
&<\sum_{m\ge 1,\ |m-k|\ge 2} a_m t^{\dd_m}
\\
&<\frac{\delta}{2} a_k t^{\dd_k}
\le \frac{5\delta}{9} a_k t^{\ee_k},
\quad t\in [t_{k-1}, t_k].
\end{split}
\]
\end{itemize}

First, remark that (iii$_\delta$') is more stringent than (iii').
So, $f_{q+sQ}$ are well-defined functions
of $z\in\dom$.
Moreover, combining estimate \eqref{e_RW_max}
and the arguments from the proof of Theorem~\ref{l_disk},
we obtain $f_{q+sQ}\in\grw(\dom)$, $s=0,1$, $q=1,\dots, Q$.

Second, we claim that
\begin{equation}\label{e_below_RW}
\frac{2}{5}e^{-h}\we(t) < \sum_{q=1}^Q |f_q(z)|+ \sum_{q=1}^Q |f_{q+Q}(z)|,
\quad t_0 < t <1.
\end{equation}

To prove the above estimate,
assume that $t\in [t_{k-1}, t_k]$
for some odd number $k$.
We have $z=t\za$ for some $\za\in\partial\dom$.
Applying \eqref{e_RW_min}, select $q(\za)$, $1\le q(\za)\le Q$
such that
\begin{equation}\label{e_qzeta}
|W_{q(\za)}[n_k] (\za)|\ge \de.
\end{equation}
Recall that $W_{q(\za)}[n]$ is a homogeneous polynomial of degree $n$.
Thus, we obtain
\begin{align*}
\frac{2\de}{5}e^{-h} \we(t)
&\overset{\textrm{(ii')}}{<}
\frac{4\de}{9}a_{k} t^{\ee_{k}}
\\
&\overset{\textrm{(iii$_\delta$')}}{<}
\de a_{k} t^{\ee_{k}} - \sum_{m\ge 1,\ |m-k|\ge 2} a_{m} t^{\ee_{m}}
\\
&\overset{\textrm{\eqref{e_RW_max} and \eqref{e_qzeta}}}{\le}
\left| \sum_{j=0}^\infty a_{2j+1} W_{q(\za)} [n_{2j+1}] (t\za) \right|
\\
&= |f_{q(\za)}(t\za)|
\\
&\le \sum_{q=1}^Q |f_q(t\za)|+ \sum_{q=1}^Q |f_{q+Q}(t\za)|.
\end{align*}

If $k$ is even, then we have an analogous estimate.
So, \eqref{e_below_RW} holds.
Putting $f_{2Q+1}\equiv 1$, we obtain the required estimate
\[
\we(\rom(z))\le C \sum_{m=1}^{2Q+1} |f_m(z)|
\]
for all $z\in\dom$.
The proof of Theorem~\ref{t_main} is finished.

\section{Applications}\label{s_appl}
It is well-known that a doubling weight function is equivalent to
a log-convex one. Hence,
Theorem~\ref{t_main} allows to recover the applications
presented in \cite[Sections~3--7]{AD12} for
the doubling weight functions $\we$.
In this section, we consider several direct applications of
Theorem~\ref{t_main}, assuming that
$\we: [0, 1)\to (0, +\infty)$ is a log-convex weight function,
unless otherwise stated.

\subsection{Weighted composition operators}\label{s_coc}
Given $g\in\hol(\dom)$ and a holomorphic mapping $\ph: \dom \to
\dom$, the weighted composition operator $\cph^g: \hol(\dom) \to
\hol(\dom)$ is defined by the formula
\[
(\cph^g f)(z) = g(z) f(\ph(z)), \quad f\in\hol(\dom),\ z\in \dom.
\]
If $g\equiv 1$, then $\cph^g$ is denoted by $\cph$ and it is called
a composition operator. Various properties of $\cph$ are presented
in monographs \cite{CmC95}, \cite{Sj93}.

Consider a linear space $\YY(\dom)$ which consists of functions
$f:\dom\to\Cbb$. We say that $\YY(\dom)$ is a \textit{lattice} if
the following implication holds:

Assume that $F\in\YY(\dom)$,
$f:\dom\to\Cbb$ is a continuous function,
and $|f(z)|\le |F(z)|$ for $z\in\dom$.
Then $f\in\YY(\dom)$.

Let $\dom$ be as in Theorem~\ref{t_main}.
The definition of a lattice and Theorem~\ref{t_main} imply the following result
(cf.~\cite[Corollary~2]{AD12}).

\begin{cory}\label{c_wco_bdd}
Suppose that $g\in\hol(\dom)$, $\ph:\dom\to\dom$ is a holomorphic
mapping and $\YY(\dom)$ is a lattice.
Then the weighted composition operator $\cph^g$ maps $\grw(\dom)$ into
$\YY(\dom)$ if and only if
\[
|g(z)| \we(\rom(\ph(z)))\in\YY(\dom).
\]
\end{cory}

\subsection{Integral operators}\label{ss_io}
Assume that $\dom$ is the unit ball $\bd$ of $\cd$.
Given $g\in \hol(\bd)$ and a holomorphic mapping $\ph: \bd \to
\bd$, the Volterra type operator $\vph^g: \hol(\bd) \to
\hol(\bd)$ is defined as
\begin{equation*}
(\vph^g f)(z) = \int_0^1 f(\ph(tz)) \frac{\rad g (tz)}{t}\, dt, \quad
f\in\hol(\bd),\ z\in \bd,
\end{equation*}
where
\[
\rad g(z) = \sum_{j=1}^d z_j \frac{\partial g}{\partial z_j} (z),
\quad z\in \bd,
\]
is the radial derivative of $g$.
If $\ph(z)\equiv z$, then $\vph^g$ is denoted by $\jg$ and it is called
an extended Ces\`aro operator (see \cite{Hu03}). In fact, if $d=1$,
then we have
\begin{equation*}
(\jg f)(z) = \int_0^z f(w)  g^\prime(w) \, dw, \quad f\in\hol(\Dbb),\
z\in \Dbb.
\end{equation*}
The above operator was introduced by Pommerenke \cite{Po77} as a natural
generalization of the classical Ces\`aro operator.
For $d=1$, various properties of the operator $\jg$ are discussed in surveys \cite{Am07}, \cite{Si06}.

Direct calculations show that
\begin{equation}\label{e_rad_volt}
\rad\vph^g f (z) = f(\ph(z)) \rad g(z), \quad z\in \bd,
\end{equation}
for all $f, g\in \hol(\bd)$ (cf.\ \cite{Hu03}). Applying (\ref{e_rad_volt})
and Corollary~\ref{c_wco_bdd}, we obtain the following fact.

\begin{cory}\label{c_volt_bdd}
Let $\YY(\bd)$ be a lattice on $\bd$.
Then the operator $\rad\vph^g$ maps $\grw(\bd)$ into
$\YY(\bd)$ if and only if
$|\rad g(z)| \we(|\ph(z)|)\in\YY(\bd)$.
\end{cory}

Given a space $X\subset\hol(\bd)$, a typical problem
is to characterize $g\in\hol(\bd)$
such that $\jg$ is a bounded operator on $X$.
Often the characterizing property is the following one:
\[
\sup_{z\in \bd} |\rad g(z)| (1-|z|) <\infty,
\]
that is, $g$ is in the Bloch space $\mathcal{B}(\bd)$
(see, e.g., \cite{AS97}, \cite{Hu03}).
Hence, it is interesting to find those $X$
for which the answer is different.
To give such examples, consider
the following exponential weight functions:
\[
\we_\al (r) = \exp\left( \frac{1}{(1-r)^\al}\right),\quad 0\le r<1,\ \al>0.
\]
Note that the above weight functions
are not doubling.

\begin{cory}\label{c_jg_exp}
Let $\al>0$ and let $g\in\hol(\bd)$.
The operator $\jg: \mathcal{A}^{\we_\al}(\bd)\to \mathcal{A}^{\we_\al}(\bd)$
is bounded if and only if
\begin{equation}\label{e_jg_exp}
\sup_{z\in\bd} |g(z)| (1-|z|)^{\al} < \infty.
\end{equation}
\end{cory}
\begin{proof}
For $h\in\hol(\Dbb)$, the norm $\|h\|_{\mathcal{A}^{\we_\al}(\Dbb)}$
is equivalent to
\[
|h(0)| + \|\rad h\|_{\mathcal{A}^{\we_\al^\prime}(\Dbb)}
\]
by \cite[Theorem~D]{PvP08}. Given $f\in\hol(\bd)$,
we apply the above fact to the slice-functions of $\jg f\in\hol(\bd)$
and conclude that
\[
\jg f \in \mathcal{A}^{\we_\al}(\bd)
\Leftrightarrow
\rad\jg f \in \mathcal{A}^{\we_\al^\prime}(\bd).
\]
By Corollary~\ref{c_volt_bdd}, the latter property
holds if and only if
\[
\sup_{z\in\bd} |\rad g(z)|\frac{\we_\al(|z|)}{\we_\al^\prime(|z|)} <\infty,
\]
that is,
\[
\sup_{z\in\bd} |\rad g(z)|(1-|z|)^{\al +1} <\infty.
\]
It remains to observe that the above inequality is equivalent to \eqref{e_jg_exp}.
\end{proof}

\begin{rem}
By definition, the weighted Bergman space
$\berg^p_{\we_\al}(\Dbb)$, $0<p<\infty$, consists of $f\in\hol(\Dbb)$
such that
\[
\int_{\Dbb} |f(z)|^p
\frac{d\nu(z)}{\we_\al(|z|)} <\infty,
\]
where $\nu$ denotes Lebesgue measure on $\Cbb$.
Formally, one may consider $\mathcal{A}^{\we_\al}(\Dbb)$ as
the weighted Bergman space $\berg^p_{\we_\al}(\Dbb)$ with $p=\infty$.
So, given $\al>0$, remark that
property \eqref{e_jg_exp} with $d=1$
also characterizes those $g\in\hol(\Dbb)$
for which the operator
$\jg: \berg^p_{\we_\al}(\Dbb) \to \berg^p_{\we_\al}(\Dbb)$, $0<p<\infty$,
is bounded (see \cite{PP10}).
\end{rem}


\subsection{Associated weights on the unit ball $\bd$}\label{s_assw}
For $\dom=\bd$, $d\ge 1$, Theorem~\ref{t_main}
is also applicable to \textit{arbitrary} weight functions via the notion
of associated weight.
Namely, the following definition
was formally introduced in \cite{BBT98}:

\begin{definition}
Given a radial weight $\ve$ on $\bd$, $d\ge 1$,
the associated weight $\tve_d$
is defined by
\[
\tve_d(z) = \sup\{|f(z)|:\ f\in\hol(\bd),\ |f|\le\ve\ \text{on}\ \bd\}.
\]
\end{definition}

As observed in \cite{BBT98}, $\tve_1$ is a radial weight,
so the associated weight function $\tve_1: [0,1)\to (0, +\infty)$ is correctly
defined.
Moreover, $\tve_1$ is known to be log-convex (see \cite{BDL99}).

Now, assume that $d\ge 2$. If $f\in\hol(\bd)$ and $|f|\le v$, then every slice-function
$f_\za(\la) = f(\la\za)$, $\za\in\spn$, $\la\in\Dbb$, is in $\hol(\Dbb)$
and $|f_\za(\la)|\le \ve(|\la|)$, $\la\in\Dbb$.
Thus, on the one hand, $\tve_d(z) \le \tve_1(|z|)$, $z\in \bd$.
On the other hand, if $f\in\hol(\Dbb)$ and $|f|\le v$ on $\Dbb$, then
$F(z_1,\dots,z_d) := f(z_1) \in \hol(\bd)$ and $|F| \le v$ on $\bd$;
hence, $\tve_d(z) \ge \tve_1(|z|)$.
So $\tve_d$ is a radial weight and the identity
$\tve_d = \tve_1 := \tve$ holds for the associated weight functions.

Clearly, $\mathcal{A}^{\ve}(\bd) = \mathcal{A}^{\tve}(\bd)$, $d\ge 1$, isometrically.
Hence, given an arbitrary weight function $\ve: [0,1)\to (0, +\infty)$,
the study of $\mathcal{A}^{\ve}(\bd)$ reduces to that of $\grw(\bd)$,
where $\we=\tve$ is a log-convex weight function.
For example, an extension of Corollary~\ref{c_wco_bdd} has the following form:

\begin{cory}\label{c_wco_bdd_arbitrary}
Let $d\ge 1$ and let $\ve: [0,1)\to (0, +\infty)$ be an arbitrary weight function.
Suppose that $g\in\hol(\bd)$, $\ph:\bd\to\bd$ is a holomorphic
mapping and $\YY(\bd)$ is a lattice.
Then the weighted composition operator $\cph^g$ maps $\mathcal{A}^{\ve}(\bd)$ into
$\YY(\bd)$ if and only if
$|g(z)| \tve(|\ph(z)|)\in\YY(\bd)$.
\end{cory}

\bibliographystyle{amsplain}

\begin{thebibliography}{10}

\bibitem{AD12}
E.~Abakumov and E.~Doubtsov, \emph{Reverse estimates in growth spaces}, Math.
  Z. \textbf{271} (2012), no.~1-2, 399--413.

\bibitem{Aab86}
A.~B. Aleksandrov, \emph{Proper holomorphic mappings from the ball to the
  polydisk}, Dokl. Akad. Nauk SSSR \textbf{286} (1986), no.~1, 11--15
  (Russian); English transl.: Soviet Math. Dokl. {\bf 33} (1) (1986), 1--5.

\bibitem{Am07}
A.~Aleman, \emph{A class of integral operators on spaces of analytic
  functions}, Topics in complex analysis and operator theory, Univ. M\'alaga,
  M\'alaga, 2007, pp.~3--30.

\bibitem{AS97}
A.~Aleman and A.~G. Siskakis, \emph{Integration operators on {B}ergman spaces},
  Indiana Univ. Math. J. \textbf{46} (1997), no.~2, 337--356.

\bibitem{BBT98}
K.~D. Bierstedt, J.~Bonet, and J.~Taskinen, \emph{Associated weights and spaces
  of holomorphic functions}, Studia Math. \textbf{127} (1998), no.~2, 137--168.

\bibitem{BDL99}
J.~Bonet, P.~Doma{\'n}ski, and M.~Lindstr{\"o}m, \emph{Essential norm and weak
  compactness of composition operators on weighted {B}anach spaces of analytic
  functions}, Canad. Math. Bull. \textbf{42} (1999), no.~2, 139--148.

\bibitem{Bo98}
A.~Borichev, \emph{The polynomial approximation property in {F}ock-type
  spaces}, Math. Scand. \textbf{82} (1998), no.~2, 256--264.

\bibitem{Con78}
J.~B. Conway, \emph{Functions of one complex variable}, second ed., Graduate
  Texts in Mathematics, vol.~11, Springer-Verlag, New York, 1978.

\bibitem{CmC95}
C.~C. Cowen and B.~D. MacCluer, \emph{Composition operators on spaces of
  analytic functions}, Studies in Advanced Mathematics, CRC Press, Boca Raton,
  FL, 1995.

\bibitem{DouCRAS}
E.~Doubtsov, \emph{Growth spaces on circular domains: composition operators and
  {C}arleson measures}, C. R. Math. Acad. Sci. Paris \textbf{347} (2009),
  no.~11-12, 609--612.

\bibitem{Ga08}
P.~Galanopoulos, \emph{On {$B_{\rm log}$} to {$Q^p_{\rm log}$} pullbacks}, J.
  Math. Anal. Appl. \textbf{337} (2008), no.~1, 712--725.

\bibitem{GX02}
P.~M. Gauthier and J.~Xiao, \emph{Bi{B}loch-type maps: existence and beyond},
  Complex Var. Theory Appl. \textbf{47} (2002), no.~8, 667--678.

\bibitem{GPPR08}
D.~Girela, J.~{\'A}. Pel{\'a}ez, F.~P\'erez-Gonz\'alez, and J.~R\"atty\"a,
  \emph{Carleson measures for the {B}loch space}, Integral Equations Operator
  Theory \textbf{61} (2008), no.~4, 511--547.

\bibitem{GPR14}
J.~Gr{\"o}hn, J.~{\'A}. Pel{\'a}ez, and J.~R{\"a}tty{\"a}, \emph{Jointly
  maximal products in weighted growth spaces}, Ann. Acad. Sci. Fenn. Math.
  \textbf{39} (2014), no.~1, 109--118.

\bibitem{Hu03}
Z.~Hu, \emph{Extended {C}es\`aro operators on mixed norm spaces}, Proc. Amer.
  Math. Soc. \textbf{131} (2003), no.~7, 2171--2179.

\bibitem{Ko07}
P.~Kot, \emph{Homogeneous polynomials on strictly convex domains}, Proc. Amer.
  Math. Soc. \textbf{135} (2007), no.~12, 3895--3903.

\bibitem{KS09}
S.~G. Krantz and S.~Stevi{\'c}, \emph{On the iterated logarithmic {B}loch space
  on the unit ball}, Nonlinear Anal. \textbf{71} (2009), no.~5-6, 1772--1795.

\bibitem{KPS82}
S.~G. Kre{\u\i}n, Yu.~{\=I}. Petun{\={\i}}n, and E.~M. Sem{\"e}nov,
  \emph{Interpolation of linear operators}, Translations of Mathematical
  Monographs, vol.~54, American Mathematical Society, Providence, R.I., 1982.

\bibitem{KP11}
E.~G. Kwon and M.~Pavlovi{\'c}, \emph{Bi{B}loch mappings and composition
  operators from {B}loch type spaces to {BMOA}}, J. Math. Anal. Appl.
  \textbf{382} (2011), no.~1, 303--313.

\bibitem{PP10}
J.~Pau and J.~{\'A}. Pel{\'a}ez, \emph{Embedding theorems and integration
  operators on {B}ergman spaces with rapidly decreasing weights}, J. Funct.
  Anal. \textbf{259} (2010), no.~10, 2727--2756.

\bibitem{PvP08}
M.~Pavlovi{\'c} and J.~{\'A}. Pel{\'a}ez, \emph{An equivalence for weighted
  integrals of an analytic function and its derivative}, Math. Nachr.
  \textbf{281} (2008), no.~11, 1612--1623.

\bibitem{Po77}
Ch. Pommerenke, \emph{Schlichte {F}unktionen und analytische {F}unktionen von
  beschr\"ankter mittlerer {O}szillation}, Comment. Math. Helv. \textbf{52}
  (1977), no.~4, 591--602.

\bibitem{RU91}
W.~Ramey and D.~Ullrich, \emph{Bounded mean oscillation of {B}loch pull-backs},
  Math. Ann. \textbf{291} (1991), no.~4, 591--606.

\bibitem{RW83}
J.~Ryll and P.~Wojtaszczyk, \emph{On homogeneous polynomials on a complex
  ball}, Trans. Amer. Math. Soc. \textbf{276} (1983), no.~1, 107--116.

\bibitem{Sj93}
J.~H. Shapiro, \emph{Composition operators and classical function theory},
  Universitext: Tracts in Mathematics, Springer-Verlag, New York, 1993.

\bibitem{Si06}
A.~G. Siskakis, \emph{Volterra operators on spaces of analytic functions---a
  survey}, Proceedings of the {F}irst {A}dvanced {C}ourse in {O}perator
  {T}heory and {C}omplex {A}nalysis, Univ. Sevilla Secr. Publ., Seville, 2006,
  pp.~51--68.

\bibitem{Xi04}
J.~Xiao, \emph{Riemann-{S}tieltjes operators on weighted {B}loch and {B}ergman
  spaces of the unit ball}, J. London Math. Soc. (2) \textbf{70} (2004), no.~1,
  199--214.

\end{thebibliography}

\end{document}